\documentclass[12pt,a4paper]{amsart}
\usepackage{graphicx}
\usepackage{color}
\usepackage{enumerate}
\usepackage{amssymb}
\usepackage[colorlinks=true]{hyperref}
\hypersetup{urlcolor=blue, citecolor=blue, draft=false}
\newtheorem{definition}{Definition}[section]
\newtheorem{theorem}{Theorem}[section]
\newtheorem{lemma}{Lemma}[section]
\newtheorem{corollary}{Corollary}[section]

\theoremstyle{remark}

\theoremstyle{remark}\newtheorem{remark}{Remark}[section]

\usepackage{subfigure}
\theoremstyle{definition}

\begin{document}
\title[Certain subclass of harmonic functions associated with univalent functions]{Certain subclass of harmonic functions associated with univalent functions}

\author[Prachi  Prajna Dash, Jugal  Kishore  Prajapat]{PRACHI PRAJNA DASH$^\ddagger$, Jugal Kishore Prajapat$^\dagger$}

\address{$^\ddagger$Department of Mathematics, Central University of Rajasthan, Bandarsindri, Kishangarh-305817, Dist.- Ajmer, Rajasthan, India}
\email{prachiprajnadash@gmail.com, jkprajapat@gmail.com}

\date{}

\begin{abstract}
In this paper, we define a subclass of sense-preserving harmonic functions associated with a class of analytic functions satisfying a differential inequality. We then establish a close relation between both subclasses. Further, we obtain some characteristic properties including radius properties, convolution, coefficient estimates and their properties, growth estimates, and convex combination for the functions in the defined subclass. At last, we produce conditions for some special functions as well as harmonic univalent polynomials to belong to the defined subclass of harmonic functions.
\end{abstract}


\maketitle
\section{Introduction}
\setcounter{equation}{0}
A complex-valued harmonic function in a domain $D\subset\mathbb{C}$ is defined as a function $f(z)=u(z)+iv(z)$, which is twice continuously differentiable and $\Delta f(z)=4f_{z\overline{z}}(z)=0$ in $D$. The function $f(z)$ has a canonical representation $f(z)=h(z)+\overline{g(z)}$ when $D$ is simply connected, where $h(z)$ and $g(z)$ are analytic functions in $D$ and are termed as analytic and co-analytic parts of $f(z)$, respectively. It is well known that $f(z)$ is locally univalent in $D$ if, and only if, the Jacobian $J_f(z)\neq0$ and $f(z)$ is sense-preserving if $J_f(z)>0$  in $D$, where $J_f(z)=|h'(z)|^2-|g'(z)|^2$. Let $\mathcal{H}$ be the class of complex-valued harmonic functions $f(z)$ defined in the open unit disk $\mathbb{D}=\{z\in \mathbb{C}: |z|<1\}$, normalized by the condition $f(0)=0=f_z(0)-1$. As a result, $h(z)$ and $g(z)$ of $f(z)$ have power series representations as
\begin{equation}\label{1.1}
h(z)=z+\sum_{n=2}^{\infty} a_n z^n \quad \text{and} \quad g(z)=\sum_{n=1}^{\infty} b_n z^n.
\end{equation}
Let $\mathcal{H}^0$ be a subclass of $\mathcal{H}$, defined as $\mathcal{H}^0=\{f(z)\in \mathcal{H} : f_{\overline{z}}(0)=0\}$ in $\mathbb{D}$. Thus, for $f(z)\in \mathcal{H}^0$, \eqref{1.1} will become
\begin{equation}\label{1.2}
h(z)=z+\sum_{n=2}^{\infty} a_n z^n \quad \text{and} \quad g(z)=\sum_{n=2}^{\infty} b_n z^n. 
\end{equation}
Let $\mathcal{S}_{\mathcal{H}}$ be another subclass of $\mathcal{H}$ containing all univalent functions of $\mathcal{H}$ and let $\mathcal{S}_{\mathcal{H}}^0$ be a subclass of $\mathcal{H}^0$ containing all univalent functions of $\mathcal{H}^0$ in $\mathbb{D}$. Clearly, $\mathcal{S}_{\mathcal{H}}^0\subset \mathcal{S}_{\mathcal{H}}$. Clunie and Sheil-Small \cite{css} started study on the class of harmonic univalent functions. These articles \cite{durh,pr,ppk,jm,jms} provide basic information about harmonic functions as well as some more recent studies. 

A starlike domain with respect to a point in it is defined to be that which contains all line segments connecting any point in it to the respected point in it. It is simply called a starlike domain if the respected point is the origin. Further, a convex domain is defined to be that which is starlike with respect to every point in it. A harmonic starlike function $f(z)\in \mathcal{H}$ is defined to be that which maps $\mathbb{D}$ to a starlike domain. The class of starlike function in $\mathcal{H}$ is denoted by $\mathcal{S}^*_{\mathcal{H}}$. In similar, a harmonic convex function $f(z)\in \mathcal{H}$ is defined to be that which maps $\mathbb{D}$ to a convex domain, and the collection of such functions is denoted by $\mathcal{K}_{\mathcal{H}}$. A close-to-convex domain is defined to be that which complement can be written as disjoint union of non-crossing half lines. A harmonic close-to-convex function $f(z)\in \mathcal{H}$ is defined to be that which maps $\mathbb{D}$ to a close-to-convex domain, and the collection of such functions is denoted by $\mathcal{C}_{\mathcal{H}}$. Let $\mathcal{K}^0_{\mathcal{H}}$,  $\mathcal{S}^{*0}_{\mathcal{H}}$ and $\mathcal{C}_{\mathcal{H}}^0$ be the subclasses of $\mathcal{K}_{\mathcal{H}}$,  $\mathcal{S}^*_{\mathcal{H}}$ and $\mathcal{C}_{\mathcal{H}}$ respectively, where $f_{\overline{z}}(0)=0$.

\subsection{Stable analytic and stable harmonic functions}
For all $|\zeta|=1$, if the analytic functions $F_{\zeta}(z)=h(z)+\zeta g(z)$ are convex, starlike, close-to-convex and univalent on $\mathbb{D}$ then the function $F(z)=h(z)+g(z)$ is said to be stable analytic convex, stable analytic starlike, stable analytic close-to-convex and stable analytic univalent on $\mathbb{D}$, respectively. In similar, if the harmonic functions $f_{\zeta}(z)=h(z)+\zeta \overline{g(z)}$ with $|\zeta|=1$ are convex, starlike, close-to-convex and univalent on $\mathbb{D}$ then the function $f(z)=h(z)+\overline{g(z)}$ is said to be stable harmonic convex, stable harmonic starlike, stable harmonic close-to-convex and stable harmonic univalent on $\mathbb{D}$, respectively. More information regarding stable harmonic functions is available in \cite{hm}. The following is one useful result of stable analytic and stable harmonic functions.

\begin{lemma}\cite{hm}\label{l2}
A sense-preserving harmonic function $f(z)=h(z)+\overline{g(z)}$ on $\mathbb{D}$ is stable harmonic univalent (respectively, convex), if and only if the analytic function $F(z)=h(z)+g(z)$ is stable analytic univalent (respectively, convex)  on $\mathbb{D}$. Same result holds for $f(z)$ to be stable harmonic starlike on $\mathbb{D}$ when the functions $f_{\zeta}(z)=h(z)+\zeta\overline{g(z)}$ fix the origin on $\mathbb{D}$ for all $|\zeta|=1$.
\end{lemma}

The convolution (or Hadamard product) of two analytic functions $F_1(z)=\sum_{n=0}^{\infty} a_n z^n $ and $F_2(z)=\sum_{n=0}^{\infty} b_n z^n $ is defined by
    \[(F_1\ast F_2)(z)=\sum_{n=0}^{\infty} a_n \,b_n \,z^n.\]
    Following the above, the convolution for two harmonic functions $f_1(z)=h_1(z)+\overline{g_1(z)}$ and $f_2(z)=h_2(z)+\overline{g_2(z)}$ in $\mathcal{H}$ is defined as
    \[(f_1\ast f_2)(z)=(h_1*h_2)(z)+\overline{(g_1\ast g_2)(z)}.\]
Clunie and Sheil-Small \cite{css} also proved a result for the convolution between a harmonic and an analytic function which states, if $f(z)$ is harmonic convex and $F(z)$ is analytic convex, then $f(z)\ast (F(z)+\gamma\overline{F(z)})$ is a harmonic close-to-convex function where $|\gamma|<1$.

Recently, Ponnusamy et al. \cite{pyy} (see also \cite{kpv}) have studied the following subclass of $\mathcal{H}$ :
\begin{align*}
    C^2_{\mathcal{H}}&=\left\{f(z)=h(z)+\overline{g(z)}\in \mathcal{H}:|h'(z)-1|<1-|g'(z)|\right\}\\
    &\subseteq C^1_{\mathcal{H}}=\left\{f(z)=h(z)+\overline{g(z)}\in \mathcal{H}:\Re(h'(z))>\Re(g'(z))\right\},
\end{align*}
and proved that its elements are close-to-convex in $\mathbb D$. Motivated by this study, we define the following subclass of $\mathcal{H}^0$ :

\begin{definition}\label{d1}
For $\lambda >0$
\[\Omega_{\mathcal{H}}^0(\lambda):= \left \{f(z)=h(z)+\overline{g(z)}\in \mathcal{H}^0 : |h(z)-zh'(z)|<\lambda-|g(z)-zg'(z)|, \; z \in \mathbb{D} \right \}.\]
\end{definition}
We see many examples of functions in $\Omega_{\mathcal{H}}^0(\lambda)$ in the next section. Now, we consider the subclass of $\mathcal{A}$ defined by us in 2023 \cite{pj}.
\begin{definition}\label{d2}
For $\lambda >0$
\[\Omega(\lambda):= \left\{ F(z)\in \mathcal{A} : |F(z)-zF'(z)|<\lambda, \; z \in \mathbb{D} \right\}.\]
\end{definition}
The class $\Omega(\lambda)$ is non-empty as the identity function satisfies the Definition \ref{d2}. Note that  $\Omega(\lambda_1) \subseteq \Omega(\lambda_2)$ when $\lambda_1 \leq \lambda_2$, and $\Omega(1/2)=\Omega$ is the class defined by Peng and Zhong \cite{peng} that is closed under convolution. The functions in $\Omega$  are univalent and starlike in $\mathbb{D}$, and convex are in $\mathbb{D}_{1/2}$ (see \cite{peng, hesam}). Obradovi\'{c} and Peng \cite{obradovic} gave certain sufficient conditions for functions belongs to the class $\Omega$, also Wani and Swaminathan \cite{wani} obtained certain sufficient conditions and subordination properties for functions belongs to the class $\Omega$. Thus, the functions in $\Omega(\lambda)$ are univalent and starlike in $\mathbb{D}$, and are convex in $\mathbb{D}_{1/2}$ for $\lambda \in \left(0, 1/2 \right]$. We \cite{pj} observed that the function
\begin{equation} \label{1.3}
F(z)=z+\dfrac{\lambda z^2}{2}+ \dfrac{\lambda z^3}{4}, \quad \lambda > 0, \, z \in \mathbb{D},
\end{equation}
belongs to the class $\Omega(\lambda)$. The following results we obtained in \cite{pj} for the class $\Omega(\lambda)$ are being used in our main results.

\begin{lemma}\label{l3}
The radius of starlikeness of $\Omega(\lambda)$ is $1/2\lambda$. Further, if $F(z)\in \Omega(\lambda)$, then
\begin{align}
|z|-\lambda |z|^2 \leq &\; |F(z)| \leq \;|z|+\lambda |z|^2,\label{1.4} \\ 
1-2\lambda |z| \leq &\; |F'(z)| \leq \;1+2\lambda |z|.\label{1.5}
\end{align}
Equality occurs in both the inequalities, if and only if $F(z)=z+\lambda \eta z^2$ for $z\in \mathbb{D}$ and $z\neq 0$, where $|\eta|=1$.
\end{lemma}

\begin{lemma}\label{l4}
Let $F(z)\in \Omega(\lambda)$. Then $F(z)$ is univalent in $|z|<1/2\lambda$, and convex in $|z|<1/4\lambda$.
\end{lemma}

\begin{lemma}\label{l5}
Let $F(z) \in \mathcal{A}$. If $|F''(z)|\leq 2\lambda$ then $F(z)\in \Omega(\lambda)$. The number $2\lambda$ is the best possible.
\end{lemma}

\begin{lemma}\label{l6}
Let $F(z)\in \mathcal{A}$. If 
\begin{equation*}
|z^2F''(z)+zF'(z)-F(z)|\leq 3\lambda,
\end{equation*}
then $F(z)\in \Omega(\lambda)$. The number $3\lambda$ is best possible.
\end{lemma}

\begin{lemma}\label{l7}
Let $F_1(z)$, $F_2(z)\in \Omega(\lambda)$. Then $(F_1\ast F_2)(z)\in\Omega(\lambda)$, if $\lambda\le 1$.
\end{lemma}

\medskip
The next section is dedicated to our main results for the class $\Omega_{\mathcal{H}}^0(\lambda)$. Here we establish a close relation between $\Omega_{\mathcal{H}}^0(\lambda)$ and $\Omega(\lambda)$. Further, we obtain some characteristic properties including radius properties, convolution, coefficient estimates and their properties, growth estimates, and convex combination for the functions in $\Omega_{\mathcal{H}}^0(\lambda)$. At last, we produce conditions for some special functions as well as harmonic univalent polynomials to belong to class $\Omega_{\mathcal{H}}^0(\lambda)$.

\medskip
\medskip
\section{Main results}
\setcounter{equation}{0}

\subsection{Association between $\Omega_{\mathcal{H}}^0(\lambda)$ and $\Omega(\lambda)$}

\begin{theorem}\label{th1}
For a harmonic function $f(z)=h(z)+\overline{g(z)}\in\Omega_{\mathcal{H}}^0(\lambda)$, there exist an analytic function $F_{\zeta}(z)=h(z)+\zeta g(z) \in \Omega(\lambda)$ and vice versa, where $|\zeta|=1$.
\end{theorem}
\begin{proof}
Let $f(z)=h(z)+\overline{g(z)}\in\Omega_{\mathcal{H}}^0(\lambda)$ and $F_{\zeta}(z)=h(z)+\zeta g(z)$, for $z\in \mathbb{D}$, where $|\zeta|=1$. Using Definition \ref{d1}, and for each $\zeta$, we can write
\begin{align*}
\left|F_{\zeta}(z)-zF'_{\zeta}(z)\right|&= \left|h(z)+\zeta g(z)-zh'(z)-z\zeta g'(z)\right| \\
&\le \left| h(z)-zh'(z) \right| + \left|g(z)-zg'(z) \right|<\lambda,
\end{align*}
for any $\lambda>0$ and $z\in \mathbb{D}$. This implies  $F_{\zeta}(z) \in \Omega(\lambda)$.

Conversely, let $F_{\zeta}(z)=h(z)+\zeta g(z) \in \Omega(\lambda)$ for any $\zeta$, where $|\zeta|<1$. Thus, for $z\in \mathbb{D}$
\[\left|F_{\zeta}(z)-zF'_{\zeta}(z)\right|=\left|h(z)+\zeta g(z)-zh'(z)-z\zeta g'(z)\right|<\lambda.\]
Now, for an appropriate choice of $\zeta$, we obtain
\[\left| h(z)-zh'(z) \right| + \left|g(z)-zg'(z) \right|<\lambda,\]
for any $\lambda>0$ and $z\in \mathbb{D}$. Therefore, $f(z)=h(z)+\overline{g(z)}\in\Omega_{\mathcal{H}}^0(\lambda)$. These complete the proof.
\end{proof}

Using Lemma \ref{l3}, Lemma \ref{l4} for a function $F_{\zeta}(z)=h(z)+\zeta g(z) \in \Omega(\lambda)$, where $|\zeta|=1$, and Theorem \ref{th1} in Lemma \ref{l2}, we yield the following results on radius properties of $\Omega_{\mathcal{H}}^0(\lambda)$.
\begin{corollary}\label{c1}
Let $f(z)=h(z)+\overline{g(z)}\in\Omega_{\mathcal{H}}^0(\lambda)$. Then $f(z)$ is stable harmonic univalent (respectively, starlike) in $|z|<1/2\lambda$. In particular, a harmonic function $f(z)\in\Omega_{\mathcal{H}}^0(\lambda)$ is univalent (respectively, starlike) in $|z|<1/2\lambda$.
\end{corollary}
\begin{corollary}\label{c2}
Let $f(z)=h(z)+\overline{g(z)}\in\Omega_{\mathcal{H}}^0(\lambda)$. Then $f(z)$ is stable harmonic convex in $|z|<1/4\lambda$. In particular, a harmonic function $f(z)\in\Omega_{\mathcal{H}}^0(\lambda)$ is convex in $|z|<1/4\lambda$.
\end{corollary}
\begin{corollary}\label{c3}
Let $f(z)=h(z)+\overline{g(z)}\in\Omega_{\mathcal{H}}^0(\lambda)$. Then $f(z)$ is stable harmonic univalent (respectively, starlike) in $\mathbb{D}$, and is stable harmonic convex in $\mathbb{D}_{1/2}$, if $\lambda\le 1/2$. In particular, a harmonic function $f(z)\in\Omega_{\mathcal{H}}^0(\lambda)$ is univalent (respectively, starlike) in $\mathbb{D}$, and is convex in $\mathbb{D}_{1/2}$, when $\lambda\le 1/2$.
\end{corollary}

Using Lemma \ref{l5} for a function $F_{\zeta}(z)=h(z)+\zeta g(z) \in \mathcal{A}$, where $|\zeta|=1$, and Theorem \ref{th1}, we yield the following result.
\begin{corollary}\label{c4}
Let $F_{\zeta}(z)=h(z)+\zeta g(z) \in \mathcal{A}$, where $|\zeta|=1$. If $|F''_{\zeta}(z)|\leq 2\lambda$ then $f(z)=h(z)+\overline{g(z)}\in\Omega_{\mathcal{H}}^0(\lambda)$. The number $2\lambda$ is the best possible.
\end{corollary}

In similar, using Lemma \ref{l6} for a function $F_{\zeta}(z)=h(z)+\zeta g(z) \in \mathcal{A}$, where $|\zeta|=1$, and Theorem \ref{th1}, we yield the following result.
\begin{corollary}\label{c5}
Let $F_{\zeta}(z)=h(z)+\zeta g(z) \in \mathcal{A}$, where $|\zeta|=1$. If 
\begin{equation*}
|z^2F_{\zeta}''(z)+zF_{\zeta}'(z)-F_{\zeta}(z)|\leq 3\lambda,
\end{equation*}
then $f(z)=h(z)+\overline{g(z)}\in\Omega_{\mathcal{H}}^0(\lambda)$. The number $3\lambda$ is best possible.
\end{corollary}

We yield the following result for convolution of functions in $\Omega_{\mathcal{H}}^0(\lambda)$ using Lemma \ref{l7} and Theorem \ref{th1}.
\begin{corollary}\label{c6}
Let $f_1(z),\; f_2(z)\in\Omega_{\mathcal{H}}^0(\lambda)$. Then $(f_1\ast f_2)(z)\in\Omega_{\mathcal{H}}^0(\lambda)$, if $\lambda \le 1$.
\end{corollary}

\begin{remark} Following Corollary \ref{c6},  if $f_2(z)=F(z)+\zeta \overline{F(z)}$, where $F(z)$ is a convex function in $\mathcal{A}$ for $|\zeta|=1$, then $f_1(z)\ast (F(z)+\zeta \overline{F(z)})\in\Omega_{\mathcal{H}}^0(\lambda)$, for $\lambda \le 1$.
\end{remark}

\medskip
\subsection{Coefficient estimates and their properties}
\begin{theorem}\label{th2}
For any $f(z)=h(z)+\overline{g(z)}\in\Omega_{\mathcal{H}}^0(\lambda)$ with the form as in \eqref{1.2} of $h(z)$ and $g(z)$, the following coefficient bounds hold true.
\[| a_n|\le\dfrac{\lambda}{n-1}\quad\text{and}\quad | b_n|\le\dfrac{\lambda}{n-1}, \]
for $n\in \mathbb{N}\setminus \{1 \}$. Sharpness of both bounds holds.
\end{theorem}
\begin{proof}
Let $f(z)\in\Omega_{\mathcal{H}}^0(\lambda)$. Note that, $zh'(z)-h(z)$ is analytic in $\mathbb{D}$ and from \eqref{1.2}, we obtain
\begin{equation}\label{2.1}
zh'(z)-h(z)=\sum_{n=2}^{\infty}(n-1) a_nz^n.
\end{equation}
Using Cauchy’s integral formula, for $z=re^{i\theta}\in \mathbb{D}$, we obtain 
\[(n-1)| a_n|=\left|\dfrac{1}{2\pi i}\int_{|z|=r}\dfrac{zh'(z)-h(z)}{z^{n+1}}dz\right|\le \dfrac{1}{2\pi}\int_{0}^{2\pi}\dfrac{\left|re^{i\theta}h'(re^{i\theta})-h(re^{i\theta})\right|}{r^{n}}d\theta ,\]
for $n\in \mathbb{N}\setminus \{1 \}$. Now using Definition \ref{d1}, for $n\in \mathbb{N}\setminus \{1 \}$
\begin{align*}
(n-1)| a_n|r^n&<\dfrac{1}{2\pi}\int_{0}^{2\pi}\left(\lambda- \left|re^{i\theta}g'(re^{i\theta})-g(re^{i\theta})\right| \right)d\theta \\
&< \lambda -\left|\dfrac{1}{2\pi}\int_{0}^{2\pi} \left( re^{i\theta}g'(re^{i\theta})-g(re^{i\theta}) \right) d\theta\right|=\lambda.
\end{align*}
Thus, for $n\in \mathbb{N}\setminus \{1 \}$
\[| a_n|< \dfrac{\lambda}{n-1},\]
as $r\rightarrow 1^-$. In similar, considering $zg'(z)-g(z)$ in place of $zh'(z)-h(z)$ and vice versa following from \eqref{2.1}, for $n\in \mathbb{N}\setminus \{1 \}$ we obtain 
\[| b_n|< \dfrac{\lambda}{n-1}.\]
Sharpness of $| a_n|$ holds for the function $f_{ a}(z)$ in $\Omega_{\mathcal{H}}^0(\lambda)$, where
\[f_{ a}(z)=z+\dfrac{\lambda}{n-1} z^n,\]
where $h(z)=z+\dfrac{\lambda}{n-1}z^n$ and $g(z)=0$, for $n\in \mathbb{N}\setminus\{1\}$. Further, sharpness of $| b_n|$ holds for the function $f_ b(z)$ in $\Omega_{\mathcal{H}}^0(\lambda)$, where
\[f_ b(z)=z-\dfrac{\lambda}{n-1}\overline{z}^n,\]
where $h(z)=z$ and $g(z)=\dfrac{-\lambda}{n-1}z^n$, for $n\in \mathbb{N}\setminus\{1\}$. This completes the proof. 
\end{proof}

\medskip
The following result provides a sufficient condition for a function to belong to class $\Omega_{\mathcal{H}}^0(\lambda)$.
\begin{theorem}\label{th3}
For any $f(z)=h(z)+\overline{g(z)}\in\mathcal{H}^0$, it also belongs to $\Omega_{\mathcal{H}}^0(\lambda)$ if
\begin{equation}\label{2.2}
\sum_{n=2}^{\infty}(n-1)(| a_n|+| b_n|)<\lambda.
\end{equation}
\end{theorem}
\begin{proof}
Let $f(z)\in\mathcal{H}^0$. Now
\begin{align*}
|zh'(z)-h(z)|\;&\le \sum_{n=2}^{\infty}(n-1)| a_n||z|^n \;\le \sum_{n=2}^{\infty}(n-1)| a_n| \\
&< \lambda - \sum_{n=2}^{\infty}(n-1)| b_n|\; \le \lambda - \sum_{n=2}^{\infty}(n-1)| b_n||z|^n \\
&= \lambda - |zg'(z)-g(z)|,
\end{align*}
which is true if \eqref{2.2} holds true. Thus, $f(z)\in\Omega_{\mathcal{H}}^0(\lambda)$, if and only if \eqref{2.2} holds true. This completes the proof.
\end{proof}

\medskip
\subsection{Growth estimates, convex combination, and characteristic properties}

\begin{theorem}\label{th5}
For any $f(z)=h(z)+\overline{g(z)}\in\Omega_{\mathcal{H}}^0(\lambda)$, we have
\begin{align}
|z|-\lambda|z|^2 &\le |f(z)| \le |z|+\lambda|z|^2, \label{2.5} \\
1-2\lambda|z| &\le |f'(z)| \le 1+2\lambda|z|. \label{2.6}
\end{align}
Here equality occurs for the function $f_3(z)=z+\lambda \eta z^2\in\Omega_{\mathcal{H}}^0(\lambda)$, where $z\in \mathbb{D}$, $z\ne 0$ and $|\eta|=1$.
\end{theorem}
\begin{proof}
Let $f(z)=h(z)+\overline{g(z)}\in\Omega_{\mathcal{H}}^0(\lambda)$ and hence using Theorem \ref{th1}, we have $F_{\zeta}(z)=h(z)+\zeta g(z)\in \Omega(\lambda)$, for all $|\zeta|=1$. Using Lemma \ref{l3}, we obtain
\begin{align*}
|z|-\lambda|z|^2 &\le |h(z)+\zeta g(z)| \le |z|+\lambda|z|^2,  \\
1-2\lambda|z| &\le |h'(z)+\zeta g'(z)| \le 1+2\lambda|z|. 
\end{align*}
Now, for appropriate choices of $\zeta$, we can write
\begin{align}
|z|-\lambda|z|^2 &\le |h(z)|-| g(z)| \quad \text{and}\quad |h(z)|+|g(z)|\le |z|+\lambda|z|^2, \label{2.7} \\
1-2\lambda|z| &\le |h'(z)|-| g'(z)| \quad \text{and}\quad |h'(z)|+|g'(z)|\le 1+2\lambda|z|, \label{2.8}
\end{align}
which further implies
\begin{align*}
|z|-\lambda|z|^2 &\le |h(z)+ \overline{g(z)}| \quad \text{and}\quad |h(z)+\overline{g(z)}|\le |z|+\lambda|z|^2,  \\
1-2\lambda|z| &\le |h'(z)+ \overline{g'(z)}| \quad \text{and}\quad |h'(z)+\overline{g'(z)}|\le 1+2\lambda|z|, 
\end{align*}
which is equivalent to \eqref{2.5} and \eqref{2.6}, respectively. As the function $f_3(z)=z+\lambda \eta z^2\in\Omega_{\mathcal{H}}^0(\lambda)$, for $z\in \mathbb{D}$, $z\ne 0$ and $|\eta|=1$, equality occurs for appropriate choices of $\eta$. This completes the proof.
\end{proof}

\medskip
\begin{theorem}\label{th6}
The class $\Omega_{\mathcal{H}}^0(\lambda)$ is closed under convex combination.
\end{theorem}
\begin{proof}
Let $f_k(z)=h_k(z)+\overline{g_k(z)}\in \Omega_{\mathcal{H}}^0(\lambda)$ for $k=1,\,2,\,3,\,\dots,\,n,$ and $\sum_{k=1}^n t_k=1$, where $t_k\in[0,1]$. Note that
\[f(z)=\sum_{k=1}^n t_k f_k(z)=h(z)+g(z),\]
is a convex combination of $f_k$'s, where $h(z)=\sum_{k=1}^n t_k h_k(z)$ and $g(z)=\sum_{k=1}^n t_k g_k(z)$. Thus, both $h(z)$ and $g(z)$ are in class $\mathcal{A}$. Now
\begin{align*}
\left|h(z)-zh'(z)\right|&=\left|\sum_{k=1}^n t_k h_k(z)-z\sum_{k=1}^n t_k h_k'(z)\right|\le \sum_{k=1}^n t_k \left| h_k(z)-z h_k'(z)\right|\\
&<\lambda- \left|\sum_{k=1}^n t_k g_k(z)-z g_k'(z)\right|=\lambda-\left|g(z)-zg'(z)\right|.
\end{align*}
Thus, $f(z)\in \Omega_{\mathcal{H}}^0(\lambda)$.
\end{proof}

\medskip
\begin{theorem}\label{th7}
For any $f(z)=h(z)+\overline{g(z)}\in\Omega_{\mathcal{H}}^0(\lambda)$, we have
\[J_f(z)\le (1+2\lambda |z|)^2.\]
Here equality occurs for the function $f_3(z)=z+\lambda \eta z^2$, where $z\in \mathbb{D}$, $z\ne 0$, and $|\eta|=1$. 
\end{theorem}
\begin{proof}
Let $f(z)\in \Omega_{\mathcal{H}}^0(\lambda)$. Now
\[\left|h(z)-zh'(z)\right|< \lambda-\left|g(z)-zg'(z)\right| \le \lambda.\]
Thus, $h(z)\in\Omega(\lambda)$. Hence using Lemma \ref{l3}, we obtain
\[|h'(z)|\le 1+ 2\lambda|z|.\]
Now
\[J_f(z)=|h'(z)|^2-|g'(z)|^2\le\;|h'(z)|^2\le \;(1+ 2\lambda|z|)^2.\]

Consider the function of $\Omega_{\mathcal{H}}^0(\lambda)$, $f_3(z)=h_3(z)+\overline{g_3(z)}=z+\lambda \eta z^2$, for $z\in \mathbb{D}$, $z\ne 0$, and $|\eta|=1$. Clearly, $h_3(z)=z+\lambda \eta z^2$. This implies 
\[|h'_3(z)|=| 1+ 2\lambda \eta z|=1+ 2\lambda|z|,\]
for an appropriate choice of $\eta$. This completes the proof.
\end{proof}

\medskip
\begin{theorem}\label{th8}
Each $f(z)=h(z)+\overline{g(z)}\in\Omega_{\mathcal{H}}^0(\lambda)$ maps $\mathbb{D}$ onto a domain bounded by a rectifiable Jordan curve, when $\lambda\le 1/2$.
\end{theorem}
\begin{proof}
Let $f(z)\in \Omega_{\mathcal{H}}^0(\lambda)$ and let $z_1$, $z_2\in \mathbb{D}$ such that $z_1\ne z_2$ and $|z_1|\ge|z_2|$. For the line segment $[z_1,z_2]$,
\begin{align}
\left|f(z_1)-f(z_2) \right|&=\left|\int_{[z_1,z_2]} \dfrac{\partial f}{\partial \tau}d\tau + \dfrac{\partial f}{\partial \overline{\tau}}d\overline{\tau}\right|\label{2.9}\\
&\le \int_{[z_1,z_2]} \left(\;|h'(\tau)|+|g'(\tau)|\; \right) |d\tau|.\notag
\end{align}
Using \eqref{2.8} (from Theorem \ref{th5}) in \eqref{2.9}, we obtain
\begin{align}
\left|f(z_1)-f(z_2) \right|&\le\int_{|z_2|}^{|z_1|} (1+2\lambda \kappa)d\kappa \label{2.10} \\
&= (|z_1|-|z_2|)+\lambda(|z_1|^2-|z_2|^2)\notag\\
&= (|z_1|-|z_2|)\left(1+\lambda(|z_1|+|z_2|)\right)\notag\\
&\le (|z_1|-|z_2|)\left(1+2\lambda\right)\le (1+2\lambda)|z_1-z_2|< \epsilon,\notag
\end{align}
for any $\epsilon>0$, such that $|z_1-z_2|<\delta:=\dfrac{\epsilon}{1+2\lambda}$, for any $\delta>0$. Therefore, $f(z)$ is uniformly continuous on $\mathbb{D}$ and consequently can be extended continuously on the boundary of $\mathbb{D}$.

Let $C$ be a curve defined by $f(e^{i\theta})$ for $\theta\in[0,2\pi]$, and let $0=\theta_0<\theta_1<\dots<\theta_n=2\pi$ be a partition of $[0,2\pi]$. Now using \eqref{2.10}, we obtain 
\begin{align*}
\sum_{k=1}^n|f(e^{i\theta_k})-f(e^{i\theta_{k-1}})|
&\le (1+2\lambda)\sum_{k=1}^n|e^{i\theta_k}-e^{i\theta_{k-1}}|\\
&\le (1+2\lambda) \left(|e^{i\theta_0}|+2(|e^{i\theta_1}|+\dots+|e^{i\theta_{n-1}}|)+|e^{i\theta_{n}}|\right)\\
&\le (1+2\lambda)2n.
\end{align*}
Thus, $C$ is rectifiable for a finite $\lambda$.

Following the proof of \cite[Lemma 3.2]{hesam}, we have $\Re(F'_{\zeta}(z))\ge 1-|z|>0$, $z\in \mathbb{D}$ for $F_{\zeta}(z)\in\Omega$. Thus, for $F_{\zeta}(z)\in\Omega(\lambda)$, $\Re(F'_{\zeta}(z))>0$ for all $z\in \mathbb{D}$ if $\lambda\le 1/2$. Therefore by the association (Theorem \ref{th1}) and from the proof of \cite[Theorem 3]{mth}\label{l8}, it is concluded that $F_{\zeta}(z)$ is univalent on the boundary of $\mathbb{D}$. Now, let's consider the points, $z_1\ne z_2$ and let $f(z_1)=f(z_2)$. The later implies
\[h(z_1)-h(z_2)=\overline{g(z_1)-g(z_2)}.\]

\textit{Case I} : Let $h(z_1)=h(z_2)$. Then $g(z_1)=g(z_2)$, and since $F_{\zeta}(z)$ is univalent, $z_1=z_2$.

\textit{Case II} : Let $h(z_1)\ne h(z_2)$, and let $\theta=\arg\{h(z_1)-h(z_2)\}\in [0,2\pi)$. Then,
\[e^{-i\theta}\left(h(z_1)-h(z_2)\right)=\overline{e^{i\theta}\left(g(z_1)-g(z_2)\right)}\in \mathbb{R}^+.\]
Taking conjugate both sides, we obtain
\[\left(h(z_1)-h(z_2)\right)=e^{2i\theta}\left(g(z_1)-g(z_2)\right),\]
that implies $F_{\zeta}(z_1)=F_{\zeta}(z_2)$, where $\zeta=e^{2i\theta}$. Hence, $z_1=z_2$. Thus, $f(z)$ is univalent on the boundary of $\mathbb{D}$. These all complete the proof.
\end{proof}

\begin{remark}
From \eqref{2.10} of Theorem \ref{th8} and for $\lambda\le 1/2$, it can be seen that functions in $\Omega_{\mathcal{H}}^0(\lambda)$ map $\mathbb{D}$ onto a domain which is contained in a disk of radius $2$ with center at the origin. Further, the area $f(\mathbb{D})$ is bounded by $4n$, where $n\in \mathbb{N}$ is finite.
\end{remark}

\medskip
\subsection{Producing conditions of belongingness to $\Omega_{\mathcal{H}}^0(\lambda)$}
Here, we produce conditions for functions to belong the class $\Omega_{\mathcal{H}}^0(\lambda)$ whose co-analytic parts include Gaussian hypergeometric functions, defined by
\begin{equation}\label{2.11}
_2F_1(a,b;c;z)=F(a,b;c;z)=\sum_{n=0}^{\infty}\dfrac{(a)_n(b)_n}{n!\,(c)_n }z^n,\quad z\in \mathbb D,
\end{equation}
for $a,\,b,\,c\in \mathbb C$, and $c\notin \mathbb Z^{-}\cup \{0\}$. Here, $(x)_n=x(x+1)(x+2)\dots(x+n-1),$ $n\in \mathbb{N}$ is the well-known Pochhammer symbol, with $(x)_0=1$. The series \eqref{2.11} converges absolutely in $\mathbb{D}$, and converges on $\overline{\mathbb{D}}$ if $\Re(c-a-b)>0$. The well-known Gauss formula \cite{nmt} is defined by
\begin{equation}\label{2.12}
F(a,b;c;1)=\sum_{n=0}^{\infty}\dfrac{(a)_n(b)_n}{n!\,(c)_n}=\dfrac{\Gamma(c)\Gamma(c-a-b)}{\Gamma(c-a)\Gamma(c-b)},
\end{equation}
for $\Re(c-a-b)>0$ and the following lemma are useful for our next result.
\begin{lemma}\cite{prf}\label{l9}
For any $a,\,b,\,c\in \mathbb{R}^{+}$, if $c-a-b-1>0$, then
\[\sum_{n=0}^{\infty}\dfrac{(n+1)(a)_n(b_n)}{n!\,(c)_n }=\left(\dfrac{ab}{c-a-b-1}+1\right)F(a,b;c;1).\]
\end{lemma}

\begin{theorem}\label{th9}
For any $a,\,b,\,c\in \mathbb{R}^{+}$ and $\eta \in \mathbb{D}$, satisfying $c-a-b-1>0$, the following results hold.
\begin{enumerate}[(a)]
\item $f_4(z)\in\Omega_{\mathcal{H}}^0(\lambda)$, if
\begin{equation}\label{2.13}
F(a,b;c;1)<\dfrac{\lambda}{|\eta|}.
\end{equation}
\item $f_5(z)\in\Omega_{\mathcal{H}}^0(\lambda)$, if
\begin{equation}\label{2.14}
\left(\dfrac{ab}{c-a-b-1}\right)F(a,b;c;1)<\dfrac{\lambda}{|\eta|}.
\end{equation}
\item $f_6(z)\in\Omega_{\mathcal{H}}^0(\lambda)$, if
\begin{equation}\label{2.15}
\left(\dfrac{ab}{c-a-b-1}+1 \right)F(a,b;c;1)<\dfrac{\lambda}{|\eta|}.
\end{equation}
\end{enumerate}
Here,
\[f_4(z)=z+\overline{\eta z \int_0^z F(a,b;c;\tau)d\tau},\;f_5(z)=z+\overline{\eta z\left(F(a,b;c;z)-1 \right)},\;f_6(z)=z+\overline{\eta z^2F(a,b;c;z)}.\]
\end{theorem}

\begin{proof}
\begin{enumerate}[(a)]
\item Let $f_4(z)=z+\overline{\eta z \int_0^z F(a,b;c;\tau)d\tau}=z+\overline{\sum_{n=2}^{\infty}I_n\,z^n}$, where
\[I_n=\eta\,\dfrac{(a)_{n-2}(b)_{n-2}}{(c)_{n-2}(n-1)!}, \qquad n\in \mathbb{N}\setminus\{1\}.\]
Now
\begin{align*}
\sum_{n=2}^{\infty}(n-1)|I_n|&=|\eta|\sum_{n=2}^{\infty}\dfrac{(a)_{n-2}(b)_{n-2}}{(c)_{n-2}(n-2)!}\\
&=|\eta|\sum_{n=0}^{\infty}\dfrac{(a)_{n}(b)_{n}}{n!\,(c)_{n}}\\
&= |\eta|\,F(a,b;c;1)<\lambda,
\end{align*}
if \eqref{2.13} holds true. Therefore, from Theorem \ref{th3}, we conclude that $f_4(z)\in\Omega_{\mathcal{H}}^0(\lambda)$. 
\\
\item Now let $f_5(z)=z+\overline{\eta z\left(F(a,b;c;z)-1 \right)}=z+\overline{\sum_{n=2}^{\infty}J_n\,z^n}$, where
\[J_n=\eta\,\dfrac{(a)_{n-1}(b)_{n-1}}{(c)_{n-1}(n-1)!}, \qquad n\in \mathbb{N}\setminus\{1\}.\]
Now
\begin{align*}
\sum_{n=2}^{\infty}(n-1)|J_n|&=|\eta|\sum_{n=2}^{\infty}\dfrac{(a)_{n-1}(b)_{n-1}}{(c)_{n-1}(n-2)!}\\
&=|\eta|\sum_{n=0}^{\infty}\dfrac{(a)_{n+1}(b)_{n+1}}{n!\,(c)_{n+1}}\\
&=|\eta|\,\dfrac{ab}{c}\sum_{n=0}^{\infty}\dfrac{(a+1)_{n}(b+1)_{n}}{n!\,(c+1)_{n}}\\
&= |\eta|\,\dfrac{ab}{c}\,\dfrac{\Gamma(c+1)\Gamma(c-a-b-1)}{\Gamma(c-a)\Gamma(c-b)}\\
&= |\eta|\left(\dfrac{ab}{c-a-b-1}\right)F(a,b;c;1)<\lambda,
\end{align*}
if \eqref{2.14} holds true. Therefore, from Theorem \ref{th3}, we conclude that $f_5(z)\in\Omega_{\mathcal{H}}^0(\lambda)$. 
\\
\item Further, let $f_6(z)=z+\overline{\eta z^2F(a,b;c;z)}=z+\overline{\sum_{n=2}^{\infty}H_n\,z^n}$, where
\[H_n=\eta\,\dfrac{(a)_{n-2}(b)_{n-2}}{(c)_{n-2}(n-2)!}, \qquad n\in \mathbb{N}\setminus\{1\}.\]
Using Lemma \ref{l9}, we obtain
\begin{align*}
\sum_{n=2}^{\infty}(n-1)|H_n|&=|\eta|\sum_{n=2}^{\infty}\dfrac{(n-1)(a)_{n-2}(b)_{n-2}}{(c)_{n-2}(n-2)!}\\
&=|\eta|\sum_{n=0}^{\infty}\dfrac{(n+1)(a)_{n}(b)_{n}}{n!\,(c)_{n}} \\
&=|\eta|\left(\dfrac{ab}{c-a-b-1}+1\right)F(a,b;c;1)<\lambda,
\end{align*}
if \eqref{2.15} holds true. Therefore, from Theorem \ref{th3}, we conclude that $f_6(z)\in\Omega_{\mathcal{H}}^0(\lambda)$. 
\end{enumerate}
\end{proof}

\medskip
Now we produce conditions for harmonic univalent polynomials to belong the class $\Omega_{\mathcal{H}}^0(\lambda)$.

\begin{theorem}\label{th10}
For any $s\in \mathbb{N}$ and $c\in \mathbb{R}^+$, the following results hold
\begin{enumerate}[(a)]
\item $p_1(z)\in\Omega_{\mathcal{H}}^0(\lambda)$, if
\begin{equation}\label{2.16}
\Gamma(c)\Gamma(c+2s)<\lambda \left(\Gamma(c+m)\right)^2.
\end{equation}
\item $p_2(z)\in\Omega_{\mathcal{H}}^0(\lambda)$, if
\begin{equation}\label{2.17}
s^2\Gamma(c)\Gamma(c+2s)<\lambda(c+2s-1) \left(\Gamma(c+m)\right)^2.
\end{equation}
\item $p_3(z)\in\Omega_{\mathcal{H}}^0(\lambda)$, if
\begin{equation}\label{2.18}
(c+s^2+2s-1)\Gamma(c)\Gamma(c+2s)<\lambda(c+2s-1) \left(\Gamma(c+m)\right)^2.
\end{equation}
\end{enumerate}
Here, $p_1(z),\,p_2(z),$ and $p_3(z)$ are harmonic univalent polynomials, given by
\begin{align}\label{2.19}
\left\{\begin{array}{r@{\;}l}
    p_1(z)&=z+\overline{\eta \sum_{n=0}^s \begin{pmatrix}s\\n
\end{pmatrix} \dfrac{(s-n+1)_n}{(c)_n}\,\dfrac{z^{n+2}}{n+1}},\\
p_2(z)&=z+\overline{\eta \sum_{n=0}^s \begin{pmatrix}s\\n
\end{pmatrix} \dfrac{(s-n+1)_n}{(c)_n}\,z^{n+1}},\\
p_3(z)&=z+\overline{\eta \sum_{n=0}^s \begin{pmatrix}s\\n
\end{pmatrix} \dfrac{(s-n+1)_n}{(c)_n}\,z^{n+2}}.  
  \end{array}\right.
\end{align}
\end{theorem}
\begin{proof}
For any $\mu\in\mathbb{C}\setminus \mathbb{Z}^{-}$ and $n\in \mathbb{N}\cap\{0\}$, we can write
\[\dfrac{(-1)^n(-\mu)_n}{n!}=\begin{pmatrix}
\mu \\ n \end{pmatrix}=\dfrac{\Gamma(\mu+1)}{n!\Gamma(\mu-n+1)}.\]
But when this $\mu=s\in\mathbb{N}$, where $s\ge n$, we can write
\[(-s)_n=\dfrac{(-1)^ns!}{(s-n)!}.\]
Using this relation on $f_4(z)$, $f_5(z)$, and $f_6(z)$ of Theorem \ref{th9}, we obtain the harmonic univalent polynomials $p_1(z),\,p_2(z),$ and $p_3(z)$, respectively as defined on \eqref{2.19}. Now putting $a=b=-s$ in the conditions \eqref{2.13}, \eqref{2.14} and \eqref{2.15}, we obtain respective belongingness conditions \eqref{2.16}, \eqref{2.17} and \eqref{2.18}, of $p_1(z),\,p_2(z),$ and $p_3(z)$. This completes the proof.
\end{proof}

\end{document}